\documentclass[12pt]{amsart}
\usepackage[all]{xy}

\usepackage{verbatim}
\usepackage{amssymb}
\usepackage{upref}
\usepackage[all]{xy}
\usepackage{color}
\usepackage{url}

\allowdisplaybreaks

\newtheorem{thm}{Theorem}[section]
\newtheorem*{thm*}{Theorem}
\newtheorem{lem}[thm]{Lemma}
\newtheorem{cor}[thm]{Corollary}
\newtheorem{prop}[thm]{Proposition}

\theoremstyle{definition}

\newtheorem*{notn*}{Notation}

\newtheorem*{hyp*}{Hypothesis}

\newtheorem*{rem*}{Remark}

\numberwithin{equation}{section}

\newcommand{\secref}[1]{Section~\textup{\ref{#1}}}

\newcommand{\thmref}[1]{Theorem~\textup{\ref{#1}}}
\newcommand{\corref}[1]{Corollary~\textup{\ref{#1}}}
\newcommand{\lemref}[1]{Lemma~\textup{\ref{#1}}}
\newcommand{\propref}[1]{Proposition~\textup{\ref{#1}}}

\newcommand{\appxref}[1]{Appendix~\textup{\ref{#1}}}

\newcommand{\midtext}[1]{\quad\text{#1}\quad}
\newcommand{\righttext}[1]{\quad\text{#1 }}
\renewcommand{\and}{\midtext{and}}
\renewcommand{\for}{\righttext{for}}

\newcommand{\KK}{\mathcal K}

\newcommand{\LL}{\mathcal L}

\newcommand{\OO}{\mathcal O}

\renewcommand{\epsilon}{\varepsilon}

\DeclareMathOperator{\aut}{Aut}

\DeclareMathOperator{\fix}{Fix}
\DeclareMathOperator{\supp}{supp}

\newcommand{\id}{\text{\textup{id}}}

\newcommand{\<}{\langle}
\renewcommand{\>}{\rangle}

\newcommand{\inv}{^{-1}}

\renewcommand{\bar}{\overline}

\newcommand{\rt}{\textup{rt}}

\newcommand{\cs}%
{\ensuremath{\mathbf{C^*}}}
\newcommand{\csnd}%
{\ensuremath{\cs\!\!_\mathbf{nd}}}
\newcommand{\coact}%
{\ensuremath{\mathbf{C^*coact}}}
\newcommand{\coactnd}%
{\ensuremath{\coact_\mathbf{nd}}}
\newcommand{\coactn}%
{\ensuremath{\coact^\mathbf{n}}}
\newcommand{\coactnnd}%
{\ensuremath{\coactn_\mathbf{nd}}}
\newcommand{\coactm}%
{\ensuremath{\coact^\mathbf{m}}}
\newcommand{\coactmnd}%
{\ensuremath{\coactm_\mathbf{nd}}}

\newcommand{\x}{\xi}
\newcommand{\y}{\eta}

\newcommand{\ideal}[2]{M(#1;#2)}
\newcommand{\cpct}[1]{#1^{(1)}}

\begin{document}

\title[Coactions and skew products]
{Coactions and skew products of topological graphs}

\author[Kaliszewski and Quigg]
{S.~Kaliszewski and John Quigg}
\address [S.~Kaliszewski] {School of Mathematical and Statistical Sciences, Arizona State University, Tempe, Arizona 85287} \email{kaliszewski@asu.edu}\
\address[John Quigg]{School of Mathematical and Statistical Sciences, Arizona State University, Tempe, Arizona 85287} \email{quigg@asu.edu}



\date{December 12, 2012}

\begin{abstract}
We show that the $C^*$-algebra of a skew-product topological 
graph $E\times_\kappa G$ is isomorphic to the crossed product of~$C^*(E)$ 
by a coaction of the locally compact group~$G$.
\end{abstract}

\maketitle

\section{Introduction}
\label{intro}

In \cite[Theorem~2.4]{kqr:graph} we proved that if $E$ is a directed graph and $\kappa$ is a function from the edges of $E$ to a discrete group $G$, then the graph algebra $C^*(E\times_\kappa G)$ of the skew-product graph is a crossed product of $C^*(E)$ by a coaction of $G$.
This was later generalized to 
homogeneous spaces $G/H$ in \cite[Theorem~3.4]{dpr},
and to higher-rank graphs in \cite[Theorem~7.1]{pqr:cover}.
In this paper we generalize the result to topological graphs and locally compact groups.
More precisely, we prove in \thmref{indirect} that if $\kappa\colon E\to G$ is a continuous
function (that is, a \emph{cocycle}), then there exists a coaction $\epsilon$
of $G$ on $C^*(E)$ such that
\[
C^*(E\times_\kappa G)\cong C^*(E)\times_\epsilon G.
\]

We give two distinct approaches to
the coaction:
in \secref{indirect section} we obtain the coaction indirectly, via an application of Landstad duality,
and
in \secref{coaction} we construct the coaction directly, applying techniques developed in \cite{KQRCorrespondenceCoaction}.
We thank Iain Raeburn for helpful conversations concerning this direct approach.

In \secref{prelim1} we record our conventions for topological graphs, $C^*$-correspondences, skew products, multiplier modules, and functoriality of Cuntz-Pimsner algebras.
In an appendix we develop a few tools that we need for dealing with certain bimodule multipliers in terms of function spaces.

\section{Preliminaries}
\label{prelim1}

In general, we refer to
\cite{rae:graphalg} (see also \cite{kat:class})
for topological graphs,
and to 
\cite{rae:graphalg, enchilada} (see also \cite{KQRCorrespondenceFunctor})
for $C^*$-correspondences,
except we make a few
minor, self-explanatory modifications.
Thus, a topological graph
$E$ comprises locally compact Hausdorff spaces $E^1,E^0$ and maps
$s,r\colon E^1\to E^0$ with $s$ a local homeomorphism and $r$ continuous.
Let $A=C_0(E^0)$, and let $X=X(E)$ be the associated
$A$-correspondence, which is the completion of $C_c(E^1)$ with
operations defined for $f\in A$ and $\x,\y\in C_c(E^1)$ by
\begin{align*}
f\cdot \x(e)&=f(r(e))\x(e)\\
\x\cdot f(e)&=\x(e)f(s(e))\\
\<\x,\y\>(v)&=\sum_{s(e)=v}\bar{\x(e)}\y(e).
\end{align*}

Throughout this paper we will also write $A'=C_0(E^1)$,
so that~$X$ can be regarded as an $A'-A$ correspondence as well as an $A$-correspondence.
Recall from \cite{kat:class} that the left $A'$-module multiplication is nondegenerate in the sense that $A'\cdot X=X$, and is determined by the homomorphism $\pi_E\colon A'\to\LL(X)$ given by $(\pi_E(f)\xi)(e)=f(e)\xi(e)$ for $f\in A'$ and $\xi\in C_c(E^1)$, and the (nondegenerate) left $A$-module multiplication $\varphi_A\colon A\to\LL(X)$ is then given by $\varphi_A(f)=\pi_E(\footnote{¥}\circ r)$ for $f\in A$.

We denote
by $(k_X,k_A)\colon(X,A)\to C^*(E)=\OO_X$ the universal Cuntz-Pimsner covariant representation,
and for any Cuntz-Pimsner covariant representation $(\psi,\pi)$ of $(X,A)$ in a $C^*$-algebra $B$
we denote
by $\psi^{(1)}\colon\KK(X)\to B$ the associated homomorphism\footnote{and here we use the notation of \cite{dkq}; Raeburn would write $(\psi,\pi)^{(1)}$} determined by $\psi^{(1)}(\theta_{\xi,\eta})=\psi(\xi)\psi(\eta)^*$,
and by $\psi\times\pi\colon C^*(E)\to B$ the unique homomorphism satisfying
\[
(\psi\times\pi)\circ k_X=\psi
\midtext{and}
(\psi\times\pi)\circ k_A=\pi.
\]
Note that in \cite{enchilada}, correspondences were called right-Hilbert bimodules, and nondegeneracy was built into the definition. All our correspondences will in fact be nondegenerate, so we can freely apply the results from \cite{enchilada}.

For skew products of topological graphs, we use a slight variation of
the definition in
\cite{dkq}:
the main difference is that
we use the same notational conventions as those in
\cite{rae:graphalg} for skew products of discrete directed graphs.
Thus, a \emph{cocycle} of a locally compact group $G$ on a topological graph $E$ is a continuous map $\kappa\colon E^1\to G$, and 
the \emph{skew product} is the topological graph $E\times_\kappa G$ with
\begin{gather*}
(E\times_\kappa G)^i=E^i\times G\quad (i=0,1),\\
r(e,t)=(r(e),\kappa(e)t),\and s(e,t)=(s(e),t).
\end{gather*}

Our conventions for multipliers of correspondences are taken primarily from 
\cite[Chapter~1]{enchilada}, but also see \cite{KQRCorrespondenceFunctor}.
If
$(\pi,\psi,\tau)\colon(A,X,B)\to (M(C),M(Y),M(D))$
is a correspondence homomorphism, then there is a unique homomorphism
$\psi^{(1)}\colon\KK(X)\to M(\KK(Y))=\LL(Y)$
such that
$\psi^{(1)}(\theta_{\xi,\eta})=\psi(\xi)\psi(\eta)^*$ for $\xi,\eta\in X$.
(For this result in the stated level of generality, in particular with no nondegeneracy assumption on $(\psi,\pi)$, see \cite[Lemma~2.1]{KQRCorrespondenceFunctor}.)
If $(\psi,\pi)$ happens to be nondegenerate, then so is $\psi^{(1)}$, and hence $\psi^{(1)}$ extends uniquely to a homomorphism
$\bar{\psi^{(1)}}\colon\LL(X)\to \LL(Y)$.

A correspondence homomorphism $(\psi,\pi)\colon(X,A)\to (M(Y),M(B))$ is defined in 
\cite{KQRCorrespondenceFunctor} to be  \emph{Cuntz-Pimsner covariant} if
\begin{enumerate}
\item $\psi(X)\subset M_B(Y)$,

\item $\pi\colon A\to M(B)$ is nondegenerate,

\item $\pi(J_X)\subset \ideal{B}{J_Y}$, and

\item the diagram
\[
\xymatrix{
J_X \ar[r]^-{\pi|} \ar[d]_{\varphi_A|}
&\ideal{B}{J_Y} \ar[d]^{\bar{\varphi_B}\bigm|}
\\
\KK(X) \ar[r]_-{\cpct\psi}
&M_B(\KK(Y))
}
\]
commutes,
\end{enumerate}
where, for an ideal $I$ of a $C^*$-algebra $C$,
\[
\ideal{C}{I}:=\{m\in M(C):mC\cup Cm\subset I\}.
\]
By \cite[Corollary~3.6]{KQRCorrespondenceFunctor},
for each Cuntz-Pimsner covariant homomorphism $(\psi,\pi)$, 
there is a unique homomorphism $\OO_{\psi,\pi}$ making the diagram
\[
\xymatrix@C+30pt{
(X,A) \ar[r]^-{(\psi,\pi)} \ar[d]_{(k_X,k_A)}
&(M_B(Y),M(B)) \ar[d]^{(\bar{k_Y},\bar{k_B})}
\\
\OO_X \ar[r]_-{\OO_{\psi,\pi}}
&M_B(\OO_Y)
}
\]
commute.
Moreover, $\OO_{\psi,\pi}$ is nondegenerate, and is injective if $\pi$ is.

Our conventions for coactions on correspondences mainly follow \cite{enchilada}, but see also \cite{KQRCorrespondenceCoaction}.

\section{Indirect approach}
\label{indirect section}

In this section we apply Landstad duality to give an indirect approach to the following result:

\begin{thm}\label{indirect}
If $\kappa\colon E^1\to G$ is a cocycle on a topological graph $E$, then there is a coaction $\epsilon$ of $G$ on $C^*(E)$ such that
\[
C^*(E\times_\kappa G)\cong C^*(E)\times_\epsilon G.
\]
\end{thm}

Throughout the rest of this paper, in addition to $A=C_0(E^0)$ and $X=X(E)$, 
we will also use the following abbreviations:
\begin{itemize}
\item $F=E\times_\kappa G$;
\item $Y=X(E\times_\kappa G)$;
\item $B=C_0((E\times_\kappa G)^0)$.
\end{itemize}

\begin{proof}
To apply Landstad duality 
\cite[Theorem~3.3]{q:land}
(stated in more modern form in \cite[Theorem~4.1]{kqr:proper}), we need the following ingredients:
an action $\alpha\colon G\to \aut C^*(F)$,
a 
$\rt-\alpha$
equivariant nondegenerate homomorphism $\mu\colon C_0(G)\to M(C^*(F))$
(where ``$\rt$'' is action of $G$ on $C_0(G)$ by right translation),
and an injective nondegenerate homomorphism
\[
\Pi\colon C^*(E)\to M(C^*(F))
\]
whose image coincides with Rieffel's generalized fixed-point algebra $C^*(F)^\alpha$.
Note that in \cite{kqr:proper}, $C^*(F)^\alpha$ would be written as $\fix(C^*(F),\alpha,\mu)$.

Since $G$ acts on the right of the skew-product topological graph $F$ via right translation in the second coordinate,
by \cite[Proposition~5.4 and discussion preceding Remark~5.3]{dkq}
we have an action $\beta=(\beta^1,\beta^0)\colon G\to \aut Y$ such that
\begin{alignat*}{2}
\beta^1_t(\xi)(e,r)&=\xi(e,rt)&\quad&\text{for }\xi\in Y\\
\beta^0_t(g)(v,r)&=g(v,rt)&&\text{for }g\in B,
\end{alignat*}
which in turn gives an action on $C^*(F)$ such that
\begin{align*}
\alpha_t\circ k_Y&=k_Y\circ \beta^1_t\\
\alpha_t\circ k_B&=k_B\circ \beta^0_t.
\end{align*}

Since $F^0=E^0\times G$, we have
\[
B=C_0(F^0)=C_0(E^0)\otimes C_0(G)=A\otimes C_0(G),
\]
so we can
define a nondegenerate homomorphism $\mu\colon C_0(G)\to M(C^*(F))$ by
\[
\mu(g)=\bar{k_B}(1_{M(A)}\otimes g),
\]
and then it is routine to verify that $\mu$ is $\rt-\alpha$ equivariant.

Finally, since the action of $G$ on $F$ is free and proper,
the proof of \cite[Theorem~5.6]{dkq}
constructs an isomorphism
\[
\Pi\colon C^*(E)\xrightarrow{\cong} C^*(F)^\alpha,
\]
and then the result follows from Landstad duality.
\end{proof}


\section{A direct approach to the coaction}
\label{coaction}

As in \secref{indirect section}, we suppose we are given a cocycle $\kappa\colon E^1\to G$ of a locally compact group $G$ on a topological graph $E$, and we continue to write 
$A=C_0(E^0)$, $A'=C_0(E^1)$, $X=X(E)$, $F=E\times_\kappa G$, $Y=X(F)$, and $B=C_0(F^0)$.

Recall that the canonical embedding $G\hookrightarrow M(C^*(G))$
is identified with a unitary element $w_G$ of $M(C_0(G)\otimes C^*(G))$.
Similarly,
we may identify $\kappa$ with a unitary element of
\[
C_b(E^1,M^\beta(C^*(G)))=M(A'\otimes C^*(G)),
\]
where $M^\beta(C^*(G))$ denotes the multiplier algebra $M(C^*(G))$ with the strict topology.

Define a nondegenerate homomorphism $\kappa^*\colon C_0(G)\to M(A')$ by $\kappa^*(f)=f\circ\kappa$, and a nondegenerate homomorphism
$\nu\colon C_0(G)\to \LL(X)$ by
\[
\nu=\pi_E\circ\kappa^*,
\]
where $\pi_E\colon A'\to \LL(X)$ is the homomorphism given on $C_c(E^1)$ by pointwise multiplication.

\begin{prop}\label{sigma}
With the above notation, 
there is a coaction 
$(\sigma,\id_A\otimes 1)$ of $G$ on $(X,A)$ defined by
\[
\sigma(\xi)=v\cdot (\xi\otimes 1),
\]
where
\[
v=\bar{\nu\otimes\id}(w_G)\in \LL(X\otimes C^*(G)),
\]
and moreover
there is a coaction
$\zeta$ of $G$ on $C^*(E)$ such that
\begin{align*}
\zeta\circ k_X&=\bar{k_X\otimes\id}\circ\sigma\\
\zeta\circ k_A&=k_A\otimes 1.
\end{align*}
\end{prop}

\begin{proof}
This follows from
\cite[Corollaries~3.4--3.5]{KQRCorrespondenceCoaction},
because
$\nu\colon C_0(G)\to \LL(X)$ commutes with $\varphi_A$.
\end{proof}

It will be convenient for us to find an equivalent expression for the coaction $\sigma$.
Note that we may regard $X$ as an $A'-A$ correspondence, and hence $X\otimes C^*(G)$ as an $(A'\otimes C^*(G))-(A\otimes C^*(G))$ correspondence.
Thus we can write
\[
\sigma(\xi)=\bar{\kappa^*\otimes\id}(w_G)\cdot (\xi\otimes 1).
\]
However, we can go further: by construction the unitary element $\bar{\kappa^*\otimes\id}(w_G)$ of $M(A'\otimes C^*(G))$ coincides with the function in $C_b(E^1,M^\beta(C^*(G)))$ whose value at an edge $e$ is
\[
\bar{\kappa^*\otimes\id}(w_G)(e)=w_G(\kappa(e))=\kappa(e);
\]
thus we can write
\[
\sigma(\xi)=\kappa\cdot(\xi\otimes 1).
\]

In \thmref{indirect} we used Landstad duality to show that $C^*(F)$ is isomorphic to the crossed product of $C^*(E)$ by a coaction $\epsilon$ of $G$; 
on the other hand, in \propref{sigma} we directly constructed a coaction~$\zeta$ of $G$ on $C^*(E)$.
To show that also $C^*(E)\times_\zeta G \cong C^*(F)$, we 
now show that in fact 
the coactions $\epsilon$ and $\zeta$ coincide.
Since the mechanism behind Landstand duality is that $\epsilon$ is pulled back along $\Pi\inv$ from the inner coaction $\delta^\mu$ on $C^*(F)$, this is accomplished by the following:

\begin{prop}
\label{equivariant}
Let $\zeta$ be the coaction on $C^*(E)$ from \propref{sigma},
and let $\Pi\colon C^*(E)\to M(C^*(F))$ 
and $\mu\colon C_0(G)\to M(C^*(F))$ be as in the proof of \thmref{indirect}.
Then $\Pi$
is $\zeta-\delta^\mu$ equivariant,
and hence $\zeta$ coincides with the coaction $\epsilon$ from \thmref{indirect}.
\end{prop}

\begin{proof}
It is equivalent to show that $(\Pi,\mu)\colon (C^*(E),C_0(G))\to M(C^*(F))$ is a covariant representation for the coaction $\zeta$, and for this we will apply \cite[Corollary~4.3]{KQRCorrespondenceCoaction}.

We will need to know how the homomorphism $\Pi$ from \cite{dkq} can be described using the techniques of \cite{KQRCorrespondenceFunctor}: \cite[Proof of Theorem~5.6]{dkq} constructs a correspondence homomorphism
\[
(\psi,\pi)\colon(X,A)\to (M_B(Y),M(B)),
\]
although the notation in \cite{dkq} is substantially different\footnote{The roles of $E,X,A$ and $F,Y,B$ are interchanged, and what we call $(\psi,\pi)$ here was written as $(\mu,\nu)$ in \cite{dkq}.}.
In the terminology of \cite[Definition~3.1]{KQRCorrespondenceFunctor}, \cite[Proof of Theorem~5.6]{dkq} shows that $(\psi,\pi)$ is Cuntz-Pimsner covariant, so that by \cite[Corollary~3.6]{KQRCorrespondenceFunctor}
there is a nondegenerate homomorphism $\OO_{\psi,\pi}$ making the diagram
\[
\xymatrix@C+30pt{
(X,A) \ar[r]^-{(\psi,\pi)} \ar[d]_{(k_X,k_A)}
&(M_B(Y),M(B)) \ar[d]^{(\bar{k_Y},\bar{k_B})}
\\
C^*(E) \ar[r]_-{\OO_{\psi,\pi}}
&M_B(C^*(F))
}
\]
commute; the homomorphism $\Pi$ from \cite{dkq} coincides with $\OO_{\psi,\pi}$.

Thus, by \cite[Corollary~4.3]{KQRCorrespondenceCoaction} it suffices to show that
\[
(\psi,\pi,\mu)\colon(X,A,C_0(G))\to (M_B(Y),M(B))
\]
is covariant for $(\sigma,\id_A\otimes 1)$, in the sense of 
\cite[Definition~2.9]{KQRCorrespondenceCoaction}.
Thus we must show that
\begin{enumerate}
\item $(\pi,\mu)$ is covariant for $(A,\id_A\otimes 1)$, and

\item \label{covariant}
$\bar{\psi\otimes\id}\circ\sigma(\xi)=\bar{\mu\otimes\id}(w_G)\cdot(\psi(\xi)\otimes 1)\cdot \bar{\mu\otimes\id}(w_G)^*$
for all $\xi\in X$.
\end{enumerate}

Condition~(i) is immediate because $\pi$ and $\mu$ commute.  
Next, we rewrite~(ii) in an equivalent form:
\begin{enumerate}
\item[(ii)$'$]
$\bar{\psi\otimes\id}(\sigma(\xi))\cdot \bar{\mu\otimes\id}(w_G)=\bar{\mu\otimes\id}(w_G)\cdot(\psi(\xi)\otimes 1)$
for all $\xi\in X$.
\end{enumerate}
To proceed further, notice that the maps $\psi$, $\pi$, and $\mu$ from \cite{dkq} take a particularly simple form in our present context:
\begin{itemize}
\item $\psi=\id_X\otimes 1_{M(C_0(G))}$;

\item $\pi=\id_A\otimes 1_{M(C_0(G))}$;

\item $\mu=1_{M(A)}\otimes \id_{C_0(G)}$.
\end{itemize}
(We should explain our notation in the above expression for $\psi$: it follows from the definitions that, as a Hilbert $(A\otimes C_0(G))$-module, $Y$ coincides with the external tensor product $X\otimes C_0(G)$ (where $C_0(G)$ is regarded as a Hilbert module over itself in the canonical way). One just has to keep in mind that $Y$ does \emph{not} coincide with $X\otimes C_0(G)$ as a $B$-correspondence --- the left $B$-module multiplication is twisted by the cocycle $\kappa$.)
Thus, for $\xi\in X$ we can write:
\begin{itemize}
\item $\psi(\xi)=\xi\otimes 1$;

\item $\bar{\psi\otimes\id}(\sigma(\xi))=\sigma(\xi)_{13}
=\kappa_{13}\cdot (\xi\otimes 1\otimes 1)$;

\item $\bar{\mu\otimes\id}(w_G)=1\otimes w_G$.
\end{itemize}

Since both sides of (ii)$'$ are adjointable Hilbert-module maps from $B\otimes C^*(G)$ to $Y\otimes C^*(G)$, and $A\odot C_c(G)$ is dense in $B$, it suffices to check that the two sides of (ii)$'$ take equal values on elementary tensors of the form $f\otimes g\otimes a$, with $f\in A$, $g\in C_c(G)$, and $a\in C^*(G)$. Evaluating the right-hand side of (ii)$'$ gives
\begin{equation}
\begin{split}
\label{LHS}
&\bigl((1\otimes w_G)\cdot (\xi\otimes 1\otimes 1)\bigr)\cdot(f\otimes g\otimes a)
\\&\quad=(1\otimes w_G)\cdot \bigl((\xi\otimes 1\otimes 1)\cdot(f\otimes g\otimes a)\bigr)
\\&\quad=(1\otimes w_G)\cdot(\xi\cdot f\otimes g\otimes a).
\end{split}
\end{equation}
Now we must use the function-space techniques from \appxref{appendix}.
We have $1\otimes w_G\in C_b(F^0,M^\beta(C^*(G)))$, with value $t$ at $(v,t)\in F^0$, and $\xi\cdot f\otimes g\otimes a\in C_c(F^1,C^*(G))$,
so by \corref{multiplier coefficient} we can evaluate the last quantity in \eqref{LHS} at $(e,t)\in F^1$, giving
\begin{equation*}
\begin{split}
&(1\otimes w_G)\bigl(r(e,t)\bigr)(\xi\cdot f\otimes g\otimes a)(e,t)
\\&\quad=(1\otimes w_G)(r(e),\kappa(e)t)(\xi\cdot f)(e)g(t)a
\\&\quad=\kappa(e)t\xi(e)f(s(e))g(t)a
\\&\quad=\xi(e)f(s(e))g(t)\kappa(e)ta.
\end{split}
\end{equation*}

We proceed similarly with the left-hand side of (ii)$'$:
\begin{equation}
\begin{split}
\label{RHS}
&\bigl(\kappa_{13}
\cdot (\xi\otimes 1\otimes 1)\cdot(1\otimes w_G)\bigr)
\cdot (f\otimes g\otimes a)
\\&\quad=\kappa_{13}\cdot\bigl(\xi\cdot f\otimes w_G(g\otimes a)\bigr)
\end{split}
\end{equation}
Now, $\kappa_{13}\in C_b(F^1,M^\beta(C^*(G)))$, 
with value $\kappa(e)$ at $(e,t)$, 
and $\xi\cdot f\otimes w_G(g\otimes a)\in C_c(F^1,C^*(G))$ because $\xi\cdot f\in C_c(E^1)$ and $w_G(g\otimes a)\in C_c(G,C^*(G))$,
so by \corref{multiplier coefficient}
we can evaluate the right-hand side of \eqref{RHS} at $(e,t)\in F^1$, giving
\begin{align*}
&\Bigl(\kappa_{13}\cdot\bigl(\xi\cdot f\otimes w_G(g\otimes a)\bigr)\Bigr)(e,t)
\\&\quad=\kappa_{13}(e,t)\bigl(\xi\cdot f\otimes w_G(g\otimes a)\bigr)(e,t)
\\&\quad=\kappa(e)(\xi\cdot f(e)\bigl(w_G(g\otimes a)\bigr)(t)
\\&\quad=\kappa(e)\xi(e)f(s(e))w_G(t)(g\otimes a)(t)
\\&\quad=\kappa(e)\xi(e)f(s(e))tg(t)a
\\&\quad=\xi(e)f(s(e))g(t)\kappa(e)ta.
\end{align*}
Therefore we have verified (ii)$'$,
and this finishes the proof.
\end{proof}

\begin{appendix}

\section{Functions and multipliers}
\label{appendix}

In \secref{coaction} we need to compute with bimodule multipliers in terms of functions. 
If $T$ is a locally compact Hausdorff space and $C$ is a $C^*$-algebra, we will use without comment the following identifications (see, e.g., \cite{tfb} or \cite[Appendix~C]{enchilada}):
\begin{itemize}
\item $C_0(T,C)=C_0(T)\otimes C$;

\item $M(C_0(T)\otimes C)=C_b(T,M^\beta(C))$,
\end{itemize}
where we write $M^\beta(C)$ to denote $M(C)$ with the strict topology.
Note that since the action of $C_b(T,M^\beta(C))$ by multipliers on $C_0(T,C)$ is via pointwise multiplication, it preserves $C_c(T,C)$.

We will need to use functions as multipliers on certain $C^*$-correspondences; since this theory is not easily available in the literature, we give details for all the results we need.
However, we make no attempt to construct a general theory --- rather, we do only enough to establish  \corref{multiplier coefficient}, which we need in \secref{coaction}.

For a topological graph $E$, we write (as in the rest of the paper) $A=C_0(E^0)$, $X=X(E)$, and $A'=C_0(E^1)$.
We will regard $X\otimes C$ both as an $(A'\otimes C)-(A\otimes C)$ correspondence and as an $(A\otimes C)$-correspondence.

The following lemma is routine:

\begin{lem}\label{first function}
$C_c(E^1,C)$ embeds densely in the $(A'\otimes C)-(A\otimes C)$ correspondence $X\otimes C$ in the following way: 
if $\xi,\eta\in C_c(E^1,C)\subset X\otimes C$,
$f\in C_c(E^1,C)\subset A'\otimes C$, 
and $g\in C_c(E^0,C)\subset A\otimes C$,
then $f\cdot \xi$ and $\xi\cdot g$ are the elements of $C_c(E^1,C)$ given by
\begin{align}
\label{one}
(f\cdot\xi)(e)&=f(e)\xi(e)
\\
\label{two}
(\xi\cdot g)(e)&=\xi(e)\cdot g(s(e)),
\end{align}
and $\<\xi,\eta\>$ is the element of $C_c(E^0,C)\subset A\otimes C$ given by
\begin{equation}
\label{three}
\<\xi,\eta\>(v)=\sum_{s(e)=v}\xi(e)^*\eta(e).
\end{equation}
Moreover, $g\cdot \xi$ is the element of $C_c(E^1,C)$ given by
\begin{equation}
\label{four}
(g\cdot\xi)(e)=g(r(e))\xi(e).
\end{equation}
\end{lem}

\begin{proof}
First of all, \eqref{two}--\eqref{three} make $C_c(E^1,C)$ into a pre-Hilbert $C_c(E^0,C)$-module (where the latter is regarded as a dense $*$-subalgebra of $C_0(E^0,C)=A\otimes C$). The only non-obvious property of pre-Hilbert modules is that \eqref{three} does give an element of $C_c(E^0,C)$, but this can be proved by an argument similar to those used in \cite[Lemma~1.5]{kat:class}.

Observe that the Hilbert-module norm on $C_c(E^1,C)$ is given by
\begin{equation}\label{norm}
\|\xi\|=\sup_{v\in E^0}\Bigl\|\sum_{s(e)=v}\xi(e)^*\xi(e)\Bigr\|^{1/2},
\end{equation}
which is larger than the uniform norm.
In particular, for $e\in E^1$ the
evaluation map
$\xi\mapsto \xi(e)$ from $C_c(E^1,C)$ to $C$
is bounded from the Hilbert-module norm to the norm of $C$.

Computing with elementary tensors of the form $\xi\otimes c$ for $\xi\in C_e(E^1)$ and $c\in C$, it is now routine to verify that
the completion of the pre-Hilbert module $C_c(E^1,C)$ is isomorphic to the external tensor product $X\otimes C$ of the Hilbert $A$-module $X$ and the Hilbert $C$-module $C$.

Now regarding $X\otimes C$ as an $(A'\otimes C)-(A\otimes C)$ correspondence,
\eqref{one} is obviously true on elementary tensors, hence for $f,\xi\in C_c(E^1)\odot C$, and therefore as stated by density of $C_c(E^1)\odot C$ in $C_c(E^1,C)$ and by continuity of evaluation.
Finally, \eqref{four} follows from \eqref{one}.
\end{proof}

\begin{lem}\label{closed}
Let $K\subset E^1$ be compact. On the subspace
\[
C_K(E^1,C):=\{\xi\in C_c(E^1,C):\supp\xi\subset K\}
\]
of $X\otimes C$, the Hilbert-module norm and the uniform norm are equivalent.
Consequently, $C_K(E^1,C)$ is norm-closed in $X\otimes C$.
\end{lem}

\begin{proof}
By \eqref{norm}, the uniform norm on $C_K(E^1,C)$ is smaller than the Hilbert-module norm from $X\otimes C$. Thus it suffices to show that the Hilbert-module norm is bounded above by a multiple of the uniform norm.
Let $\xi\in C_K(E^1,C)$.
Using compactness of $K$ and local homeomorphicity of $s$, it is easy to verify that the cardinalities of the intersections $K\cap s\inv(v)$ for $v\in E^0$ are bounded above by some nonnegative integer $d$.
Then for any $v\in E^0$ we have
\begin{align*}
\Bigl\|\sum_{s(e)=v}\xi(e)^*\xi(e)\Bigr\|
&\le \sum_{s(e)=v}\|\xi(e)\|^2
\le d\|\xi\|_u^2,
\end{align*}
where $\|\x\|_u$ denotes the uniform norm of $\xi$,
and the result follows.
\end{proof}

Since $X\otimes C$ is a nondegenerate $(A'\otimes C)-(A\otimes C)$ correspondence, the left module action of $A'\otimes C$ extends canonically to the multiplier algebra $M(A'\otimes C)=C_b(E^1,M^\beta(C))$ (and similarly for the left module action of $A\otimes C$). The following corollary allows us to compute this extended left module action on
generators:

\begin{cor}
\label{multiplier coefficient}
If $m\in C_b(E^1,M^\beta(C))$ and $\xi\in C_c(E^1,C)\subset X\otimes C$, then the element $m\cdot\xi$ of $X\otimes C$ lies in $C_c(E^1,C)$, and
\begin{equation}\label{left}
(m\cdot\xi)(e)=m(e)\xi(e)\for e\in E^1.
\end{equation}
If $n\in C_b(E^0,M^\beta(C))$ then both $n\cdot \xi$ and $\xi\cdot n$ lie in $C_c(E^1,C)$, and
\begin{align*}
(n\cdot\xi)(e)&=n(r(e))\xi(e);
\\
(\xi\cdot n)(e)&=\xi(e)n(s(e)).
\end{align*}
\end{cor}

\begin{proof}
Choose a net $\{m_i\}$ in $C_c(E^1,C)$ converging strictly to $m$ in $M(C_0(E^1,C))$.
The conclusion holds for each $m_i\cdot\xi$, by \lemref{first function}.
Let $K=\supp\xi$, a compact subset of $E^1$.
Then $m_i\cdot\xi\in C_K(E^1,C)$ for all $i$.
Since $m_i\cdot\xi\to m\cdot\xi$ in the norm of $X\otimes C$ ,
we have $m\cdot\xi\in C_K(E^1,C)$, by \lemref{closed}.

For each $e\in E^1$,
by norm-continuity of evaluation on $C_c(E^1,C)$
we have
\[
(m\cdot\xi)(e)
=\lim_i (m_i\cdot\xi)(e)
=\lim_i m_i(e) \xi(e).
\]
Moreover, for any $a\in C$, 
we can choose $f\in C_0(E^1,C)$ such that $f(e)=a$ and compute:
\begin{align*}
m_i(e)a
&=m_i(e)f(e)
=(m_if)(e)
\to (mf)(e)
=m(e)a.
\end{align*}
Thus evaluation is strictly continuous on $C_b(E^1,M^\beta(C))$;
in particular,
\begin{align*}
\lim_i m_i(e)\xi(e)
&=m(e)\xi(e),
\end{align*}
which establishes~\eqref{left}.

The statement for $n\cdot \xi$ follows by composing with the range map $r\colon E^1\to E^0$, and the statement for $\xi\cdot n$ is proved similarly to the above argument for $m\cdot \xi$.
\end{proof}

\end{appendix}



\providecommand{\bysame}{\leavevmode\hbox to3em{\hrulefill}\thinspace}
\providecommand{\MR}{\relax\ifhmode\unskip\space\fi MR }
\providecommand{\MRhref}[2]{%
  \href{http://www.ams.org/mathscinet-getitem?mr=#1}{#2}
}
\providecommand{\href}[2]{#2}

\end{document}